\documentclass{amsart}
\usepackage[utf8]{inputenc}

\usepackage{amssymb}
\usepackage{amsmath}
\usepackage{amsthm}
\usepackage[pdftex]{graphicx}
\usepackage{tikz}
\usepackage{url}
\usepackage{mathtools}
\usepackage[backend=biber,style=alphabetic,sorting=nty,doi=false,isbn=false,url=false,eprint=true]{biblatex}
\usepackage{geometry}
\renewbibmacro{in:}{}
\usetikzlibrary{positioning}
\usetikzlibrary{calc}
\usetikzlibrary{decorations.pathreplacing}
\usetikzlibrary{cd}

\newcommand{\PP}{\mathbb{P}}

\newcommand{\R}{\mathbb{R}}

\newcommand{\boldsig}{\boldsymbol{\Sigma}}
\newcommand{\boldpi}{\boldsymbol{\Pi}}
\newcommand{\bolddelta}{\boldsymbol{\Delta}}

\newcommand\forces{\Vdash}

\newcommand{\OR}{\mathbin{\text{or}}}
\newcommand{\AND}{\mathbin{\text{and}}}
\newcommand{\GP}{\mathsf{GP}}
\newcommand{\non}{\operatorname{non}}
\newcommand{\cov}{\operatorname{cov}}
\newcommand{\add}{\operatorname{add}}
\newcommand{\cof}{\operatorname{cof}}
\newcommand{\nul}{\mathsf{null}}
\newcommand{\meager}{\mathsf{meager}}

\newcommand{\frakc}{\mathfrak{c}}
\newcommand{\frakb}{\mathfrak{b}}
\newcommand{\frakd}{\mathfrak{d}}
\newcommand{\all}{\mathsf{all}}
\newcommand{\ZFC}{\mathsf{ZFC}}
\newcommand{\ZF}{\mathsf{ZF}}
\newcommand{\AD}{\mathsf{AD}}
\newcommand{\CH}{\mathsf{CH}}
\newcommand{\proj}{\operatorname{proj}}
\newcommand{\supt}{\operatorname{supt}}
\newcommand{\cf}{\operatorname{cf}}
\newcommand{\LT}{\mathbf{LT}}
\newcommand{\nBC}{\mathbf{nBC}}
\newcommand{\Det}{\operatorname{Det}}
\newcommand{\Coll}{\mathrm{Coll}}
\newcommand{\Con}{\operatorname{Con}}
\newcommand{\projective}{\mathsf{projective}}
\newcommand{\nullset}{\mathsf{nullset}}
\newcommand{\Borel}{\mathsf{Borel}}
\newcommand{\I}{\mathcal{I}}
\newcommand{\IP}{\mathsf{IP}}

\DeclarePairedDelimiter\abs{\lvert}{\rvert}
\newcommand{\seq}[1]{{\langle#1\rangle}}
\DeclarePairedDelimiterX{\norm}[1]{\lVert}{\rVert}{#1}

\renewcommand\emptyset{\varnothing}
\renewcommand\subset{\subseteq}
\renewcommand{\setminus}{\smallsetminus}

\theoremstyle{definition}
\newtheorem{thm}{Theorem}[section]
\newtheorem*{thm*}{Theorem}
\newtheorem{defi}[thm]{Definition}
\newtheorem{notation}[thm]{Notation}
\newtheorem*{defi*}{Definition}
\newtheorem{lem}[thm]{Lemma}
\newtheorem*{lem*}{Lemma}
\newtheorem{fact}[thm]{Fact}
\newtheorem*{fact*}{Fact}
\newtheorem{prop}[thm]{Proposition}
\newtheorem*{prop*}{Proposition}

\newtheorem*{rmk*}{Remark}
\newtheorem{cor}[thm]{Corollary}
\newtheorem*{cor*}{Corollary}
\newtheorem{prob}[thm]{Problem}

\title{Goldstern's principle about unions of null sets}
\author{Tatsuya Goto}
\date{\today}
\address{
	\newline
	Institute of Discrete Mathematics and Geometry, TU Wien \newline
	Wiedner Hauptstrasse 8-10/104, 1040 Wien, Austria
}
\email{tatsuya.goto@tuwien.ac.at}
\thanks{Supported by JSPS KAKENHI Grant Number JP22J20021}

\begin{document}
	\begin{abstract}
		Goldstern showed in \cite{Goldstern1993} that the union of a real-parametrized, monotone family of Lebesgue measure zero sets has also Lebesgue measure zero provided that the sets are uniformly $\boldsig^1_1$.
		Our aim is to study to what extent we can drop the $\boldsig^1_1$ assumption.
		We show Goldstern's principle for the pointclass $\boldpi^1_1$ holds.
		We show that Goldstern's principle for the pointclass of all subsets is consistent with $\ZFC$ and show its negation follows from $\CH$.
		Also we prove that Goldstern's principle for the pointclass of all subsets holds both under $\ZF + \AD$ and in Solovay models.
	\end{abstract}
	
	\maketitle
	
	%\tableofcontents
	
	\section{Introduction}
	
	In \cite{Goldstern1993}, Goldstern showed the following theorem: let $\seq{A_x : x \in \omega^\omega}$ be a family of Lebesgue measure zero sets. Assume that this family is increasing in the sense that if $x, x' \in \omega^\omega$ satisfies $(\forall n\in\omega)(x(n) \le x'(n))$ then $A_x \subset A_{x'}$. Also assume that $A = \{ (x, y) : y \in A_x \}$ is a $\boldsig^1_1$ set. Then $\bigcup_{x \in \omega^\omega} A_x$ has also Lebesgue measure zero.
	Goldstern stated this theorem in terms of complements and applied it in the context of uniform distribution theory.
	Our main focus is to study to what extent we can remove this $\boldsig^1_1$ assumption.
	
	In Section \ref{sec:review}, we will review Goldstern's proof and define our principle $\GP(\Gamma)$, which is the principle that replaces $\boldsig^1_1$ in Goldstern's theorem with a pointclass $\Gamma$.
	In Section \ref{sec:gppi11}, we will show $\GP(\boldpi^1_1)$.
	In Section \ref{sec:neggpall}, we will show the consistency of $\neg \GP(\all)$, where $\all$ is the pointclass of all subsets of Polish spaces.
	In Section \ref{sec:gpall}, we will prove that $\GP(\all)$ is consistent with $\ZFC$.
	In Section \ref{sec:gpdelta12}, we will show that $\GP(\bolddelta^1_2)$ implies a combinatorial hypothesis.
	In Section \ref{sec:determinacy}, we will show that $\GP(\all)$ holds under $\ZF+\AD$.
	In Section \ref{sec:largecardinals}, we will prove under large cardinal hypotheses that $\GP(\boldsig^1_n)$ and $\GP(\bolddelta^1_{n+1})$ can be separated for each $n \ge 2$.
	In Section \ref{sec:solovaymodel}, we will show that $\GP(\all)$ holds in Solovay models.
	In Section \ref{sec:otherideals}, we will show $\neg \GP(\Borel, \mathcal{E})$, where $\mathcal{E}$ is the $\sigma$-ideal generated by closed null sets.
	
	%In this paper, we always assume $\ZF$ and the countable choice axiom for the reals.
	%The axiom of choice is assumed in Sections 3 to 4 and 6 to 8, but not in the other sections. In particular, we don't assume it in Sections 2 and 5.
	
	In the rest of the section, we review basic terminology.
	
	\begin{defi}
		Define relations $\le$ and $\le^*$ on $\omega^\omega$ as follows: for $x, x' \in \omega^\omega$,
		\begin{align*}
			x \le x' &\iff (\forall n\in\omega)(x(n) \le x'(n))\text{, and} \\
			x \le^* x' &\iff (\exists m \in \omega)(\forall n \ge m)(x(n) \le x'(n)).
		\end{align*}
	\end{defi}
	
	The cardinals defined below are typical examples of what are called cardinal invariants. For a detailed explanation of these cardinal invariants, see \cite{bartoszynski1995set}.
	
	\begin{defi}
		\begin{enumerate}
			\item We say $F \subset \omega^\omega$ is an unbounded family if $\neg(\exists g \in \omega^\omega)(\forall f \in F)(f \le^* g)$. Put $\frakb = \min \{\abs{F} : F \subset \omega^\omega \text{ unbounded family} \}$.
			\item We say $F \subset \omega^\omega$ is an dominating family if $(\forall g \in \omega^\omega)(\exists f \in F)(g \le^* f)$. Put $\frakd = \min \{\abs{F} : F \subset \omega^\omega \text{ dominating family} \}$.
			\item $\nul$ and $\meager$ denote the Lebesgue measure zero ideal and Baire first category ideal on $2^\omega$, respectively.
			\item For an ideal $\I$ on a set $X$: $\add(\I)$ (the additivity number of $\I$) is the smallest cardinality of a family $F$ of sets in $\I$ such that the union of $F$ is not in $\I$.
			\item For an ideal $\I$ on a set $X$: $\cov(\I)$ (the covering number of $\I$) is the smallest cardinality of a family $F$ of sets in $\I$ such that the union of $F$ is equal to $X$.
			\item For an ideal $\I$ on a set $X$: $\non(\I)$ (the uniformity of $\I$) is the smallest cardinality of a subset $A$ of $X$ such that $A$ does not belong to $\I$.
			\item For an ideal $\I$ on a set $X$: $\cof(\I)$ (the cofinality of $\I$) is the smallest cardinality of a family $F$ of sets in $\I$ that satisfies the following condition: for every $A \in \I$, there is $B \in F$ such that $A \subset B$.
		\end{enumerate}
	\end{defi}

	\begin{notation}
		We denote interpretations of Borel codes by the hat symbol. So $\hat{c}$ is the interpretation of $c$ if $c$ is a Borel code.
	\end{notation}

	In this paper we will often use pointclasses from projective hierarchy, such as $\boldsig^1_1$,  $\boldsig^1_2$,  $\Sigma^1_1(a)$, etc. For a detailed explanation of projective hierarchy, see  \cite{moschovakis2009descriptive}.

\section{Review of Goldstern's proof}\label{sec:review}

In \cite{Goldstern1993}, Goldstern proved the following theorem.
In the proof, he uses the Shoenfield absoluteness theorem and the random forcing. As for these, see \cite[Chapter 3]{kanamori2008higher}.

\begin{thm}[Goldstern]\label{goldsterntheorem}
	Let $(Y, \mu)$ be a Polish probability space.
	Let $A \subset \omega^\omega \times Y$ be a $\boldsig^1_1$ set.
	Assume that for each $x \in \omega^\omega$,
	\[
	A_x := \{ y \in Y : (x, y) \in A \}
	\]
	has measure $0$.
	Also, assume that $(\forall x, x' \in \omega^\omega)(x \le x' \Rightarrow A_{x} \subset A_{x'})$.
	Then $\bigcup_{x \in \omega^\omega} A_x$ also has measure $0$.
\end{thm}
\begin{proof}
	We may assume that $Y = 2^\omega$ and $\mu$ is the usual measure of $2^\omega$ since for every Polish probability space $Y$, there is a Borel isomorphism between measure $1$ subsets of $Y$ and $2^\omega$ that preserves measure. 
	
	Fix a defining formula and a parameter of $A$.
	In generic extensions if we write $A$, we refer to the set defined by the formula and the parameter in the model.
	
	Since $A$ and $\bigcup_{x \in \omega^\omega} A_x$ are $\boldsig^1_1$ sets, they are Lebesgue measurable.
	Toward a contradiction, assume that $B := \bigcup_{x \in \omega^\omega} A_x$ does not have measure 0. Then $B$ has positive measure.
	By inner regularity of the measure, we can take a closed set $K \subset B$ with positive measure.
	Take a Borel code $k$ of $K$.
	We take a random real $r \in 2^\omega$ over $V$ such that $r \in \hat{k}$.
	
	Now for each $x \in \omega^\omega \cap V$, we have $r \not \in A_x$.
	In order to prove it, take a Borel code $d_x$ such that $A_x \subset \hat{d_x}$ and $\mu(\hat{d_x}) = 0$.
	But the condition $A_x \subset \hat{d_x}$ is $\boldpi^1_1$. Thus, since the random real avoids $\hat{d_x}$, we have $r \not \in A_x$.
	
	Therefore we have
	\[
	r \not \in \bigcup_{x \in \omega^\omega \cap V} A_x.
	\]
	But in $V[r]$, it also holds that
	\[
	(\forall x, x' \in \omega^\omega) (x \le x' \rightarrow A_x \subset A_{x'})
	\]
	since this formula is $\boldpi^1_2$.
	Thus, by the assumption that $A_x$ is increasing and the fact that the random forcing is $\omega^\omega$-bounding, this implies
	\[
	r \not \in \bigcup_{x \in \omega^\omega} A_x.
	\]
	Therefore, in $V[r]$, it holds that
	\[
	(\exists r' \in 2^\omega)(r' \in \hat{k} \setminus B)
	\]
	because $r' = r$ suffices.
	Recalling $B$ is a $\boldsig^1_1$ set, this statement is written by a $\boldsig^1_2$ formula.
	Therefore, by Shoenfield's absoluteness, it holds also in $V$. That is, there exists an $r' \in 2^\omega$ in $V$ such that
	\[
	r' \in K \setminus B.
	\]
	This contradicts the choice of $K$.
\end{proof}

We define the principle $\GP(\Gamma)$.
We call the condition $(\forall x, x' \in \omega^\omega)(x \le x' \Rightarrow A_{x} \subset A_{x'})$ the monotonicity condition for $A$.

\begin{defi}
	Let $\Gamma$ be a pointclass.
	Then $\GP(\Gamma)$ means the following statement:
	Let $(Y, \mu)$ be a Polish probability space  and $A \subset \omega^\omega \times Y$ be in $\Gamma$.
	Assume the monotonicity condition for $A$.
	Also suppose that $A_x$ has $\mu$-measure $0$ for every $x \in \omega^\omega$.
	Then $\bigcup_{x \in \omega^\omega} A_x$ also has $\mu$-measure $0$.
	
	We define $\GP^*(\Gamma)$ as the version of $\GP(\Gamma)$ where the order $\le$ is replaced by $\le^*$.
\end{defi}

By Goldstern's theorem, we have $\GP(\boldsig^1_1)$.

For the reasons stated in the proof of Theorem \ref{goldsterntheorem}, if the pointclass $\Gamma$ contains all Borel sets and closed under Borel functions, then we may assume that the space $(Y, \mu)$ in the definition of $\GP(\Gamma)$ is the Cantor space.

\begin{thm}\label{thm:measurabilityandgp}
	For every natural number $n$, if $\boldsig^1_{n+1}$-$\mathbb{B}$-absoluteness holds and every $\boldsig^1_n$ set is Lebesgue measurable, then $\GP(\boldsig^1_n)$ holds.
	In particular, if every $\boldsig^1_2$ set is Lebesgue measurable, then $\GP(\boldsig^1_2)$ holds.
\end{thm}
\begin{proof}
	This is proved by the same argument as Theorem \ref{goldsterntheorem}.
	Recall that $\boldsig^1_3$-$\mathbb{B}$-absoluteness follows from $\boldsig^1_2$ measurability (see \cite[Theorem 9.2.12 and 9.3.8]{bartoszynski1995set}).
\end{proof}

Clearly $\GP(\Gamma)$ implies $\GP^*(\Gamma)$.
If we make a slight assumption on the pointclass $\Gamma$, then the converse holds.
We only consider such pointclasses.

\begin{lem}
	If a pointclass $\Gamma$ is closed under recursive substitution and projection along $\omega$, then
	$\GP^*(\Gamma) \Rightarrow \GP(\Gamma)$.
\end{lem}
\begin{proof}
	Assume that $A \in \Gamma$ and for each $x \in \omega^\omega$, $A_x$ has $\mu$-measure $0$ and that $(\forall x, x' \in \omega^\omega)(x \le x' \Rightarrow A_{x} \subset A_{x'})$.
	Put $B_x = \bigcup \{ A_y : x \text{ and } y \text{ are almost equal} \}$.
	
	Then by assumption, $B \in \Gamma$ and for each $x \in \omega^\omega$, $B_x$ has $\mu$-measure $0$ and that $(\forall x, x' \in \omega^\omega)(x \le^* x' \Rightarrow B_{x} \subset B_{x'})$.
	Therefore, by $\GP^*(\Gamma)$, $\bigcup_{x \in \omega^\omega} B_x$ is a measure $0$ set.
	Thus $\bigcup_{x \in \omega^\omega} A_x$ is a measure $0$ set.
\end{proof}

\section{$\GP(\boldpi^1_1)$}\label{sec:gppi11}

In this section, we prove that $\GP(\boldpi^1_1)$ holds.

\begin{fact}[{{\cite{kechris1973337, tanaka196711}}}]\label{fact:kechris}
	Let $U \in \Sigma^1_1$, $U \subset \omega^\omega \times 2^\omega$ be the universal for $\boldsig^1_1$ subset of $2^\omega$.
	Then the relation $\mu(U_x) > r$ for $x \in \omega^\omega$ and $r \in \R$ is $\Sigma^1_1$.
\end{fact}

\begin{cor}\label{cor:kechris}
	Let $A \subset \omega^\omega \times 2^\omega$ be a $\boldsig^1_1$ set.
	Then the relation $\mu(A_x) > r$ for $x \in \omega^\omega$ and $r \in \R$ is $\boldsig^1_1$.
\end{cor}
\begin{proof}
	Take universal sets $U$ and $U^{(2)}$ for $\boldsig^1_1$ subsets of $2^\omega$ and $\omega^\omega \times 2^\omega$, respectively, with the following coherent property:
	$U(S(e, x), y) \iff U^{(2)}(e, x, y)$,
	where $S$ is a recursively continuous function.
	As for existence of such coherent universal sets, see \cite[Section 3.H]{moschovakis2009descriptive}.	Take $e \in \omega^\omega$ such that $A(x, y) \iff U^{(2)}(e, x, y)$.
	Then we have 
	\begin{align*}
		\mu(A_x) > r \iff \mu(U_{S(e, x)}) > r,
	\end{align*}
	which is a $\boldsig^1_1$ relation.
\end{proof}

\begin{cor}\label{cor:axisnullissig11}
	Let $A \subset \omega^\omega \times 2^\omega$ be a $\boldpi^1_1$ set.
	Then the relation $\mu(A_x) = 0$ for $x \in \omega^\omega$ is $\boldsig^1_1$.
\end{cor}
\begin{proof}
	Let $B = (\omega^\omega \times 2^\omega) \setminus A$, which is $\boldsig^1_1$ set. We have
	\begin{align*}
		\mu(A_x) = 0 \iff \mu(B_x) = 1 \iff (\forall n)(\mu(B_x) > 1 - 1/2^n),
	\end{align*}
	which is a $\boldsig^1_1$ relation.
\end{proof}

\begin{thm}
	$\GP(\boldpi^1_1)$ holds.
\end{thm}
\begin{proof}
	Let $A \subset \omega^\omega \times 2^\omega$ be a $\boldpi^1_1$ set.
	Assume $\seq{A_x : x \in \omega^\omega}$ is monotone and each $A_x$ is null.
	Take a Laver real $d$ over $V$.
	$(\forall x \in \omega^\omega)(\mu(A_x) = 0)$ holds in $V$ and this sentence is $\boldpi^1_2$ using Corollary \ref{cor:axisnullissig11}.
	So in $V[d]$, $\mu(A_d) = 0$ holds.
	Also monotonicity of $\seq{A_x : x \in \omega^\omega}$ can be written as a $\boldpi^1_2$ formula and holds in $V$, so it holds also in $V[d]$.
	Since $d$ is a dominating real over $V$, we have $(\bigcup_{x \in \omega^\omega} A_x)^V \subset \bigcup_{x \in \omega^\omega \cap V} A_x \subset A_d$. Therefore $(\bigcup_{x \in \omega^\omega} A_x)^V$ is null in $V[d]$.
	Since Laver forcing preserves Lebesgue outer measure, it holds that $\bigcup_{x \in \omega^\omega} A_x$ is null in $V$.
\end{proof}	

\section{Consistency of $\neg \GP(\all)$}\label{sec:neggpall}

\begin{defi}
	We call a sequence $\seq{A_\alpha : \alpha < \kappa}$ a \textit{null tower} if it is an increasing sequence of measure $0$ sets such that its union does not have measure $0$.
\end{defi}

\begin{thm}\label{nulltowerimpliesneggp}
	If there is a null tower of length either $\frakb$ or $\frakd$, then $\neg \GP(\all)$ holds.
\end{thm}
\begin{proof}
	In the case of $\frakb$: 
	By assumption, we take an increasing sequence $\seq{A_\alpha : \alpha < \frakb}$ of measure-0 sets such that $\bigcup_{\alpha < \frakb} A_\alpha$ doesn't have measure $0$.
	We can take an increasing and unbounded sequence $\seq{x_\alpha : \alpha < \frakb }$ with respect to $\le^*$. (This sequence is not necessarily cofinal.)
	For each $x \in \omega^\omega$, put
	\[
	\alpha(x) = \min \{ \alpha < \frakb : x_\alpha \not \le^* x  \}.
	\]
	This is well-defined since $\seq{x_\alpha : \alpha < \frakb }$ is unbounded.
	And then put
	\[
	B_x = A_{\alpha(x)}.
	\]
	
	Now each $B_x$ has measure $0$ and we have
	\begin{align*} 
		x \le x' &\Rightarrow x \le^* x' \\
		&\Rightarrow (\forall \alpha)(x_\alpha \le^* x \Rightarrow x_\alpha \le^* x') \\
		&\Rightarrow \{ \alpha : x_\alpha \not \le^* x' \} \subset \{ \alpha : x_\alpha \not \le^* x \} \\
		&\Rightarrow \alpha(x) \le \alpha(x') \\
		&\Rightarrow B_x \subset B_{x'}.
	\end{align*}
	Thus $\seq{B_x : x \in \omega^\omega}$ is monotone.
	Also we have $\bigcup_{x \in \omega^\omega} B_x = \bigcup_{\alpha < \frakb} A_\alpha$.
	Indeed, it is obvious that the left-hand side is contained in the right-hand side.
	To prove the reverse inclusion, it is sufficient to each $A_\alpha$ is contained in some $B_x$.
	So fix $\alpha$ and consider $x = x_{\alpha}$.
	Since the sequence $\seq{x_\alpha : \alpha < \frakb }$ is increasing, we have $\alpha \le \alpha(x)$.
	Thus $A_\alpha \subset A_{\alpha(x)} = B_x$.
	
	Therefore, $\bigcup_{x \in \omega^\omega} B_x$ doesn't have measure $0$.
	
	In the case of $\frakd$:
	As above, we can take an increasing sequence $\seq{A_\alpha : \alpha < \frakd}$ of measure-0 sets such that $\bigcup_{\alpha < \frakd} A_\alpha$ doesn't have measure $0$.
	By the definition of $\frakd$, we can take a dominating sequence $\seq{x_\alpha : \alpha < \frakd }$ with respect to $\le^*$. (This sequence is not necessarily increasing.)
	
	For each $x \in \omega^\omega$, put
	\[
	\alpha(x) = \min \{ \alpha < \frakd : x \le^* x_\alpha  \}
	\]
	and put
	\[
	B_x = A_{\alpha(x)}.
	\]
	One can easily show that $\seq{B_x : x \in \omega^\omega}$ is monotone.
	Also we have $\bigcup_{x \in \omega^\omega} B_x = \bigcup_{\alpha < \frakd} A_\alpha$.
	That the left-hand side is contained in the right-hand side is obvious.
	To show the reverse inclusion,  fix $\alpha$.
	Since the sequence $\seq{x_\beta : \beta < \alpha}$ is not a dominating family, we can find an $x \in \omega^\omega$ such that for all $\beta < \alpha$, $x \not \le^* x_\beta$.
	Then $\alpha \le \alpha(x)$.
	Thus we have $A_\alpha \subset A_{\alpha(x)} = B_x$.
\end{proof}

\begin{cor}\label{neggpall}
	Assume that at least one of the following three conditions holds:
	\begin{enumerate}
		\item $\add(\nul) = \frakb$,
		\item $\non(\nul) = \frakb$ or
		\item $\non(\nul) = \frakd$.
	\end{enumerate}
	Then $\neg \GP(\all)$ holds.
	In particular the continuum hypothesis implies $\neg \GP(\all)$.
\end{cor}
\begin{proof}
	Clearly there are null towers of length both $\add(\nul)$ and $\non(\nul)$.
	So using Theorem \ref{nulltowerimpliesneggp}, we have this corollary.
\end{proof}

\begin{prop}\label{prop:gpandaddm}
	$\GP(\all)$ implies $\add(\meager) < \cof(\meager)$.
\end{prop}
\begin{proof}
	Assume $\add(\meager) = \cof(\meager)$.
	Let $\seq{M_\alpha : \alpha < \kappa}$ be a cofinal increasing sequence of meager sets. We can take such a sequence since $\add(\meager) = \cof(\meager) = \kappa$.
	For each $\alpha < \kappa$, take $x_\alpha \in M_{\alpha+1} \setminus M_\alpha$.
	
	Now recall from Rothberger's theorem, there is a Tukey morphism $(\varphi, \psi) \colon (2^\omega, \nul, \in) \to (\meager, 2^\omega, \not \ni)$. That is, there are $\varphi \colon 2^\omega \to \meager$ and $\psi \colon 2^\omega \to \nul$ such that $\varphi(x) \not \ni y$ implies $x \in \psi(y)$ for every $x, y \in 2^\omega$.
	
	Using this theorem, we put $N_\alpha = \bigcap_{\beta \ge \alpha} \psi(x_\beta)$ for $\alpha < \kappa$.
	Then $\seq{N_\alpha : \alpha < \kappa}$ is a sequence of null sets of length $\kappa = \frakb$ and its union is $2^\omega$.
\end{proof}

In the following proposition, we show that $\GP(\all)$ cannot be forced by finite support iteration of ccc forcings.

\begin{prop}\label{prop:gpandcccfsi}
	For every finite support iteration of ccc forcings $\seq{P_\alpha : \alpha < \nu}$ with $\cf(\nu) \ge \aleph_1$, we have $P_\nu \forces \neg \GP(\all)$.
\end{prop}
\begin{proof}
	Let $G$ be a $(V, P_\nu)$ generic filter and work in $V[G]$.
	Let $\seq{c_\alpha : \alpha < \cf(\nu) }$ be a sequence of Cohen reals added cofinally by $P_\alpha$.
	
	For a Cohen real $c$, let $\nullset(c)$ denote the standard null set constructed from $c$.
	
	We have the following:
	\begin{itemize}
		\item For every $x \in \omega^\omega$, there is $\alpha < \cf(\nu)$ such that for every $\beta > \alpha$ we have $c_\beta \not <^* x$.
		\item For every $z \in 2^\omega$, there is $\alpha < \cf(\nu)$ such that for every $\beta > \alpha$ we have $z \in \nullset(c_\beta)$.
	\end{itemize}
	
	For $x \in \omega^\omega$, we let $\alpha_x = \min \{ \alpha : (\forall \beta > \alpha) (c_\beta \not <^* x) \}$ and let $A_x = \bigcap_{\beta > \alpha_x} \nullset(c_\beta)$.
	
	We can easily show that each $A_x$ is a null set, the sequence $\seq{A_x : x \in \omega^\omega}$ is monotone and the union $\bigcup_{x \in \omega^\omega} A_x$ is equal to $2^\omega$.
	Therefore, $\seq{A_x : x \in \omega^\omega}$ is a witness of $\neg \GP(\all)$.
\end{proof}

\section{Consistency of $\GP(\all)$}\label{sec:gpall}

In this section, as in the previous section, we assume $\ZFC$.
To obtain a model of $\GP(\all)$, $\add(\nul) \ne \frakb$, $\non(\nul) \ne \frakb$, $\non(\nul) \ne \frakd$ and $\add(\meager) \ne \cof(\meager)$ need to hold.
A natural model in which they hold is the Laver model.
In this section, we will see that $\GP(\all)$ actually holds in the Laver model.

\begin{thm}\label{nulltowerequiv}
	Assume that $\frakb = \frakd$ and let both of these be $\kappa$.
	Then the following are equivalent.
	\begin{enumerate}
		\item There is a null tower of length $\kappa$.
		\item $\neg \GP(\all)$.
	\end{enumerate}
\end{thm}
\begin{proof}
	That (1) implies (2) is shown in Theorem \ref{nulltowerimpliesneggp}.
	
	We now prove the converse implication.
	Assume that $\neg \GP^*(\all)$.
	Then we can take $A \subset \omega^\omega \times 2^\omega$ such that each section $A_x$ has measure $0$ and $(\forall x, x' \in \omega^\omega)(x \le^* x' \Rightarrow A_{x} \subset A_{x'})$ holds and $B = \bigcup_{x \in \omega^\omega} A_x$ does not have measure $0$.
	By $\frakb = \frakd = \kappa$, we can take a cofinal increasing sequence $\seq{x_\alpha : \alpha < \kappa }$ with respect to $\le^*$.
	For each $\alpha < \kappa$, put $C_\alpha = A_{x_\alpha}$.
	Then each $C_\alpha$ has measure $0$.
	Since $\alpha \mapsto x_\alpha$ is increasing and $x \mapsto A_x$ is monotone, $\seq{ C_\alpha : \alpha < \kappa}$ is also increasing.
	Also, since $\seq{x_\alpha : \alpha < \kappa }$ is cofinal, we have $B = \bigcup_{\alpha < \kappa} C_\alpha$.
	So $\bigcup_{\alpha < \kappa} C_\alpha$ does not have measure $0$.
	Thus $\seq{ C_\alpha : \alpha < \kappa}$ is a null tower of length $\kappa$.
\end{proof}

The following lemma and theorem requires knowledge of proper forcing. See \cite{goldstern1992tools}.

\begin{lem}\label{lem:reflection}
	Assume $\CH$.
	Let $\seq{P_\alpha, \dot{Q}_\alpha : \alpha < \omega_2}$ be a countable support iteration of proper forcing notions such that
	\[
	\forces_\alpha \abs{\dot{Q}_\alpha} \le \frakc \ \ (\text{for all }\alpha < \omega_2).
	\]
	Let $\seq{\dot{X}_\alpha : \alpha < \omega_2}$ be a sequence of $P_{\omega_2}$-names such that
	\[
	\forces_{\omega_2} (\forall \alpha < \omega_2) (\dot{X}_\alpha \text{ has measure $0$}).
	\]
	Then the set 
	\begin{align*}
		C = \{ \alpha < \omega_2 : & \cf(\alpha) = \omega_1 \,\AND \\
		&\forces_{\omega_2} (\seq{\dot{X}_\beta \cap V[\dot{G}_\alpha] : \beta < \alpha } \in V[\dot{G}_\alpha] \AND (\forall \beta < \alpha) (\dot{X}_\beta \cap V[\dot{G}_\alpha]  \text{ has measure $0$})^{V[\dot{G}_\alpha]}) \}.
	\end{align*}
	contains a $\omega_1$-club set in $\omega_2$.
	
\end{lem}
\begin{proof}
	This is an example of a reflection argument. See also \cite[Chapter 26]{halbeisen2012combinatorial}.
	
	Take a sequence $\seq{\dot{c}_\beta : \beta < \omega_2}$ of names of Borel codes such that
	\[
	\forces_{\omega_2} (\forall \beta < \omega_2) (\dot{X}_\beta \subset \hat{\dot{c}}_\beta \AND \hat{\dot{c}}_\beta \text{ has measure $0$}).
	\]
	For each $\beta < \omega_2$, take $\gamma_\beta < \omega_2$ such that $\dot{c}_\beta$ is a $P_{\gamma_\beta}$-name.
	
	Since for each $\alpha < \omega_2$, $\forces_\alpha \CH$, we can take a sequence $\seq{\dot{x}^\alpha_i : i < \omega_1}$ such that \[
	\forces_\alpha ``\seq{\dot{x}^\alpha_i : i < \omega_1}\text{ is an enumeration of } 2^\omega"\].	
	For each $\alpha, \beta < \omega_2$ and $i < \omega_1$, take a maximal antichain $A^{\alpha,\beta}_i$ such that
	\[
	A^{\alpha,\beta}_i \subset \{p \in P_{\omega_2} : p \forces \dot{x}^\alpha_i \in \dot{X}_\beta \OR p \forces \dot{x}^\alpha_i \not \in \dot{X}_\beta \}.
	\]
	Since $P_{\omega_2}$ has $\omega_2$-cc, we can take $\delta^{\alpha,\beta}_i < \omega_2$ such that
	\[
	\bigcup \{ \supt(p) : p \in A^{\alpha,\beta}_i \} \subset \delta^{\alpha,\beta}_i.
	\]
	
	We define a function $f$ from $\omega_2$ into $\omega_2$ as follows:
	\[
	f(\nu) = \sup \left( \{ \gamma_\beta : \beta \le \nu \} \cup \{ \delta^{\alpha,\beta}_i : \alpha, \beta \le \nu, i < \omega_1 \}  \right)
	\]
	Put
	\[
	C' = \{ \alpha < \omega_2 : \cf(\alpha) = \omega_1, (\forall \nu < \alpha) f(\nu) < \alpha \}.
	\]
	Then clearly $C'$ is $\omega_1$-club set. So it suffices to show that $C' \subset C$.
	
	Let $\alpha \in C'$ and we shall prove $\alpha \in C$.
	Fix $\beta < \alpha$.
	Define a $P_\alpha$-name $\dot{Y}$ by
	\begin{align*}
		\forces_\alpha ``\dot{Y} = \bigcup_{\alpha' < \alpha} \{ \dot{x}^{\alpha'}_i  : &(p \forces \dot{x}^{\alpha'}_i \in \dot{X}_\beta)^V \\
		& \text{for some } p \in A^{\alpha',\beta}_i \AND p \upharpoonright \alpha \in \dot{G} \}" \tag{$\ast$}.
	\end{align*}
	
	We claim that $\forces_{\omega_2} \dot{X}_\beta \cap V[\dot{G}_\alpha] = \dot{Y}$.
	In order to prove this, take a $(V, P_{\omega_2})$-generic filter $G$.
	In $V[G]$, take $x \in \dot{X}^G_\beta \cap V[G_\alpha]$.
	Since no new real is added at stage $\alpha$, we can take $\alpha' < \alpha$ such that $x \in V[G_{\alpha'}]$. Thus there is $i < \omega_1$ such that $x = (\dot{x}^{\alpha'}_i)^G$.
	Since $(\dot{x}^{\alpha'}_i)^G \in \dot{X}^G_\beta$, in $V$, we can take a $p \in G \cap A^{\alpha',\beta}_i$ such that $p \forces \dot{x}^{\alpha'}_i \in \dot{X}_\beta$.
	We have $p \in A^{\alpha',\beta}_i$.
	Thus $x$ is an element of $\dot{Y}^G$.
	
	Conversely, take an element $x$ of $\dot{Y}^G$.
	So we can take $\alpha' < \alpha$, $i < \omega_1$ and $p \in P_{\omega_2}$ such that
	\[
	x = (\dot{x}^{\alpha'}_i)^G, (p \forces \dot{x}^{\alpha'}_i \in \dot{X}_\beta)^V, p \in A^{\alpha',\beta}_i \AND p \upharpoonright \alpha \in G_\alpha.
	\]
	Clearly we have $x \in V[G_{\alpha'}] \subset V[G_\alpha]$.
	Suppose that $(\dot{x}^{\alpha'}_i)^G \not \in \dot{X}^G_\beta$.
	Then we can take $q \in G$ such that $q \forces \dot{x}^{\alpha'}_i \not \in \dot{X}_\beta$.
	By the maximality of $A^{\alpha',\beta}_i$, we can take $r \in A^{\alpha',\beta}_i \cap G$.
	Since both $q$ and $r$ are elements of $G$, $q$ and $r$ are compatible.
	So $r \forces \dot{x}^{\alpha'}_i \not \in \dot{X}_\beta$.
	Thus $p$ and $r$ are incompatible.
	But $\supt(p), \supt(r) \subset \alpha$. So $p \upharpoonright \alpha$ and $r \upharpoonright \alpha$ are incompatible. But they are elements of $G_\alpha$. It contradicts that $G_\alpha$ is a $(V, P_\alpha)$-generic filter.
	
	Thus we have $\forces_{\omega_2} \dot{X}_\beta \cap V[G_\alpha] \in V[G_\alpha]$.
	
	By performing the above operations simultaneously with respect to the $\beta$, we have \[\forces_{\omega_2} \seq{\dot{X}_\beta \cap V[\dot{G}_\alpha] : \beta < \alpha } \in V[\dot{G}_\alpha].\]
	
	Since we have $\forces_{\omega_2} \dot{X}_\beta \subset \hat{\dot{c}}_\beta$, it holds that
	\[
	\forces_{\omega_2} ``\dot{X}_\beta \cap V[G_\alpha] \subset \hat{\dot{c}}_\beta  \text{ has measure $0$}".
	\]
	Therefore, we have $\alpha \in C$.
\end{proof}

Recall that $\LT$ denotes the Laver forcing.
As for basic properties of Laver forcing, see \cite{bartoszynski1995set}.

\begin{thm}\label{lavermodelthm}
	
	Assume $\CH$. Let $\seq{P_\alpha, \dot{Q}_\alpha : \alpha < \omega_2}$ be the countable support iteration such that
	\[
	\forces_\alpha \dot{Q}_\alpha = \LT \ \ (\text{for all }\alpha < \omega_2).
	\]
	Then 
	\[
	\forces_{\omega_2} \GP(\all).
	\]
	In particular, if $\ZFC$ is consistent then so is $\ZFC + \GP(\all)$.
\end{thm}
\begin{proof}
	By Theorem \ref{nulltowerequiv} and the fact that $
	\forces_{\omega_2} \frakb = \frakd = \omega_2$, it is sufficient to show that 
	\[
	\forces_{\omega_2} ``\text{There is no null tower of length $\omega_2$}".
	\]
	Let $G$ be a $(V, P_{\omega_2})$-generic filter.
	In $V[G]$, consider an increasing sequence  $\seq{A_\alpha : \alpha < \omega_2}$ of measure $0$ sets.
	By Lemma \ref{lem:reflection}, we can find a stationary set $S \subset \omega_2$ such that for all $\alpha \in S$, $\cf(\alpha) = \omega_1$ and
	\[
	(\seq{A_\beta \cap V[G_\alpha] : \beta < \alpha } \in V[G_\alpha] \AND (\forall \beta < \alpha)((A_\beta \cap V[G_\alpha]  \text{ has measure $0$})^{V[G_\alpha]}).
	\]
	Fix $\alpha \in S$.
	Put $B_\alpha := \bigcup_{\beta < \alpha} A_\beta \cap V[G_\alpha]$.
	Then we have $\bigcup_{\alpha < \omega_2} B_\alpha = \bigcup_{\alpha < \omega_2} A_\alpha$.
	We now prove that $B_\alpha$ is also a measure $0$ set in $V[G_\alpha]$.
	Let $\alpha'$ be the successor of $\alpha$ in $S$.
	Then $B_\alpha$ is a measure $0$ set in $V[G_{\alpha'}]$.
	Since the quotient forcing $P_{\alpha'} / G_\alpha$ is a countable suppport iteration of the Laver forcing, this forcing preserves positive outer measure.
	So $B_\alpha$ is also a measure $0$ set in $V[G_{\alpha}]$.
	
	For each $\alpha \in S$, take a Borel code $c_\alpha \in \omega^\omega$ of a measure $0$ set such that $B_\alpha \subset \hat{c}_\alpha$ in $V[G_\alpha]$.
	Since $\cf(\alpha) = \omega_1$, each $c_\alpha$ appears a prior stage.
	Then by Fodor's lemma, we can take a stationary set $S' \subset \omega_2$ that is contained by $S$ and $\beta < \omega_2$ such that $(\forall \alpha \in S')(c_\alpha \in V[G_\beta])$.
	But the number of reals in $V[G_\beta]$ is $\aleph_1$, so we can take $T \subset S'$ unbounded in $\omega_2$ and $c$ such that $(\forall \alpha \in T)(c_\alpha = c)$.
	Then we have $\bigcup_{\alpha < \omega_2} A_\alpha \subset \hat{c}$ in $V[G]$.
	So $\bigcup_{\alpha < \omega_2} A_\alpha$ has measure $0$.
\end{proof}

\begin{cor}
	$\Con(\ZFC) \rightarrow \Con(\ZFC+\GP(\projective)+\neg \GP(\all))$.
	Here, $\projective = \bigcup_{n \ge 1} \boldsig^1_n$.
\end{cor}
\begin{proof}
	Assume $\CH$ and let $P$ be the forcing poset from Theorem \ref{lavermodelthm}, that is the countable support iteration of Laver forcing notions of length $\omega_2$.
	Then we have $P \forces \GP(\all)$.
	In particular we have $P \forces \GP(\projective)$.
	Let $\dot{Q}$ be a $P$-name of the poset
	\[
	\operatorname{Fn}(\omega_1, 2, \omega_1) = \{ p : \text{$p$ is a countable partial function from $\omega_1$ to $2$} \}
	\]
	with the reverse inclusion order.
	It is well-known that provably $\operatorname{Fn}(\omega_1, 2, \omega_1)$ adds no new reals and forces $\CH$.
	So we have $P \ast \dot{Q} \forces \CH$.
	Since $P \forces \GP(\projective)$ and $P \forces ``\dot{Q} \text{ adds no new reals}"$, we have $P \ast \dot{Q} \forces \GP(\projective)$.
\end{proof}

\section{A necessary condition for $\GP(\bolddelta^1_2)$}\label{sec:gpdelta12}

As mentioned in Section \ref{sec:review}, a sufficient condition for $\GP(\boldsig^1_2)$ is every $\boldsig^1_2$ set is Lebesgue measurable, or equivalently for every real $a$, there is an amoeba real over $L[a]$. (This equivalence was proved by Solovay, see \cite[Theorem 9.3.1]{bartoszynski1995set}).
In this section we give a necessary condition for $\GP(\bolddelta^1_2)$.

\begin{fact}[{Spector--Gandy, see \cite[Propositon 4.4.3]{chong2015recursion}}]\label{spectorgandy}
	Let $A$ be a set of reals.
	Then $A$ is a $\Sigma^1_2$ set iff there is a $\Sigma_1$ formula $\varphi$ such that
	\[
	x \in A \iff (L_{\omega_1^{L[x]}}[x], \in) \models \varphi(x).
	\]
\end{fact}

The following is well-known.

\begin{lem}\label{amoebaimpdominating}
	Let $M$ be a model of $\ZFC$ contained by $V$.
	And assume that the set $\{ y \in 2^\omega : y \text{ is a random real over } M  \}$ has measure $1$.
	Then there is a dominating real over $M$.
\end{lem}
\begin{proof}
	Let $\nBC$ denote the set of all Borel codes for measure $0$ Borel sets.
	There is an absolute Tukey morphism $(\varphi, \psi)$ that witnesses $\add(\nul) \le \frakb$.
	That is, $(\varphi, \psi)$ satisfies $\varphi \colon \omega^\omega \to \nBC$, $\psi \colon \nBC \to \omega^\omega$, and $(\forall x \in \omega^\omega)(\forall y \in \nBC) (\hat{\varphi(x)} \subset \hat{y} \rightarrow x \le^* \psi(y))$.
	By absoluteness, if $x \in \omega^\omega \cap M$, then we have $\varphi(x) \in M$.
	Now
	\[
	\bigcup_{x \in \omega^\omega \cap M} \hat{\varphi(x)}
	\]
	has measure 0 since this is contained in $\{ y \in 2^\omega : y \text{ is not a random real over } M  \}$ by the definition of randomness.
	Take a $z \in \nBC$ such that
	\[
	\bigcup_{x \in \omega^\omega \cap M} \hat{\varphi(x)} \subset \hat{z}.
	\]
	Now put $w = \psi(z)$.
	Then using the fact that $(\varphi, \psi)$ is Tukey morphism, we have $w$ is a dominating real over $M$.
\end{proof}

\begin{thm}
	For every real $a$, $\GP(\Delta^1_2(a))$ implies there is a dominating real over $L[a]$.
	Thus, $\GP(\bolddelta^1_2)$ implies for every real $a$, there is a dominating real over $L[a]$.
	In particular $V=L$ implies $\neg \GP(\Delta^1_2)$.
\end{thm}
\begin{proof}
	Fix a real $a$.
	Assume that $L[a] \cap \omega^\omega$ is unbounded.
	Note that, in this situation, we have $\omega_1^{L[a]} = \omega_1$.
	Let $\seq{x_\alpha : \alpha < \omega_1}$ be a cofinal increasing sequence in  $\omega^\omega \cap L[a]$. We can take this sequence with a $\Delta_1(a)$ definition by using a $\Delta_1(a)$ canonical wellordering of $L[a] \cap \omega^\omega$.
	Note that this sequence is unbounded in $V \cap \omega^\omega$ by assumption.
	
	Take a sequence $\seq{c_\alpha : \alpha < \omega_1}$ consisting of all Borel codes for measure $0$ Borel sets in $L[a]$.
	As above, we can take this sequence with a $\Delta_1(a)$ definition.
	
	For each $x \in \omega^\omega$, put
	\[
	\alpha(x) = \min \{ \alpha : x_\alpha \not \le^* x \}.
	\]
	This is well-defined since $\seq{x_\alpha : \alpha < \omega_1}$ is unbounded in $V \cap \omega^\omega$.
	Also put
	\[
	A_x = \bigcup_{\beta < \alpha(x)} \hat{c_\beta}.
	\]
	Then the set $A$ is $\Delta^1_2(a)$, by Spector--Gandy theorem and the following equations:
	\begin{align*}
		A &= \{ (x, y) \in \omega^\omega \times 2^\omega : (\exists \beta < \alpha(x))\ y \in \hat{c_\beta} \} \\
		&= \{ (x, y) \in \omega^\omega \times 2^\omega : (\exists \alpha)(x_\alpha \not \le^* x \AND (\forall \beta < \alpha)(x_\beta \le^* x) \AND (\exists \beta < \alpha)(y \in \hat{c_\beta})) \} \\
		&= \{ (x, y) \in \omega^\omega \times 2^\omega : (\forall \alpha)((x_\alpha \not \le^* x \AND (\forall \beta < \alpha)(x_\beta \le^* x)) \rightarrow (\exists \beta < \alpha)(y \in \hat{c_\beta})) \}.
	\end{align*}
	
	Note that each $A_x$ ($x \in \omega^\omega$) is a measure $0$ set since it is a countable union of measure $0$ sets.
	And we can easily observe that $x \le^* x'$ implies $A_x \subset A_{x'}$.
	
	Since $\alpha \le \alpha(x_\alpha)$, we have $\bigcup_{x \in \omega^\omega} A_x = \bigcup_{\alpha < \omega_1} \hat{c_\alpha}$.
	Thus it is sufficient to show that $C := \bigcup_{\alpha < \omega_1} \hat{c_\alpha}$ is not a measure $0$ set.
	In order to show this, assume that $C$ is a measure $0$ set.
	Note that if a real $y \in 2^\omega$ does not belong to $C$, then $y$ is a random real over $L[a]$ since the sequence $\seq{c_\alpha : \alpha < \omega_1}$ enumerates all measure $0$ Borel codes in $L[a]$.
	So since we assumed $C$ is measure $0$, the set $\{ y \in 2^\omega : y \text{ is a random real over } L[a]  \}$ has measure $1$.
	Thus by Lemma \ref{amoebaimpdominating}, there is a dominating real over $L[a]$.
	This contradicts the assumption.
	
	So we constructed a set $A$ that violates $\GP(\bolddelta^1_2)$. This finishes the proof.
\end{proof}

Therefore, we obtain the following diagram of implications.

\[
\begin{tikzpicture}
	\node (sigma12lm) at (0, 0) {$\boldsig^1_2$-LM};
	\node (sigma12bp) at (3, 0) {$\boldsig^1_2$-BP};
	\node[align=left] (dominating) at (8, 0) {$(\forall a \in \R)(\exists z \in \omega^\omega)$ \\ \hspace{0.5cm}$(\text{$z$ is a dominating real over $L[a]$})$};
	\node (gpsigma12) at (1.5, -2) {$\GP(\boldsig^1_2$)};
	\node (gpdelta12) at (4.5, -2) {$\GP(\bolddelta^1_2$)};
	
	\draw[->] (sigma12lm) -> (sigma12bp);
	\draw[->] (sigma12bp) -> (dominating);
	\draw[->] (sigma12lm) -> (gpsigma12);
	\draw[->] (gpsigma12) -> (gpdelta12);
	\draw[->] (gpdelta12) -> (dominating);
\end{tikzpicture}
\]

Here, $\boldsig^1_2$-LM means $\boldsig^1_2$-Lebesgue measurability and $\boldsig^1_2$-BP means $\boldsig^1_2$-Baire property.
$\boldsig^1_2$-LM and $\GP(\boldsig^1_2)$ can be separated since the Laver model over $L$ satisfies $\GP(\all)$ but not $\boldsig^1_2$-LM.

\section{Consequences of determinacy}\label{sec:determinacy}

In this section, we don't assume the axiom of choice and we will discuss a consequence of determinacy for Goldstern's principle.

\begin{thm}\label{thm:determinacyimpliesgp}
	Let $\Gamma$ be a pointclass that contains all Borel subsets and is closed under Borel substitution.
	Assume $\Det(\Gamma)$.
	Then $\GP(\proj(\Gamma))$ holds, where $\proj(\Gamma)$ is the pointclass of all projections along $\omega^\omega$ of a set in $\Gamma$.
	
	In particular, $\AD$ implies $\GP(\all)$. Also $\Det(\boldpi^1_n)$ implies $\GP(\boldsig^1_{n+1})$ for every $n \ge 1$.
\end{thm}
\begin{proof}
	This proof is based on Harrington's covering game. See also \cite[Exercise 6A.17]{moschovakis2009descriptive}.
	In this proof, we use the following notation: for $j < n < \omega$,
	\[\proj^n_j : (\omega^\omega)^n \to (\omega^\omega)^{n-1}; (x_0, \dots, x_{n-1}) \mapsto (x_0, \dots, x_{j-1}, x_{j+1}, \dots, x_{n-1}).\]
	Fix $B \subset \omega^\omega \times \omega^\omega \times 2^\omega$ and $A = \proj^3_0(B)$ such that each section $A_x$ has measure $0$, $(\forall x, x' \in \omega^\omega)(x \le x' \Rightarrow A_{x} \subset A_{x'})$.
	Also let $\varepsilon > 0$.
	We have to show that the outer measure $\mu^*(\proj(A))$ is less than or equal to $\varepsilon$.
	
	Fix a Borel isomorphism $\pi \colon 2^\omega \to \omega^\omega$.
	Consider the following game:
	At stage $n$, player I plays $(s_n, t_n, u_n) \in \{0, 1\}^3$. Player II then plays a finite union $G_n$ of basic open sets such that $\mu(G_n) \le \varepsilon / 16^{n+1}$.
	In this game, we define that player I wins if and only if $(z, x, y) \in B$ and $y \not \in \bigcup_{n \in \omega} G_n$, where $x = \pi(s_0, s_1, \dots), y = (t_0, t_1, \dots)$ and $z = \pi(u_0, u_1, \dots)$.
	
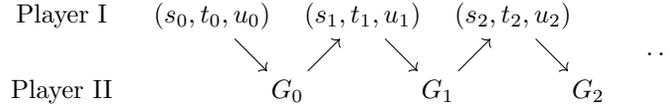
	\[
	\begin{tikzpicture}
		\node (playerI) at (0, 0) {Player I};
		\node (playerII) at (0, -1) {Player II};
		\node (I0) at (2, 0) {$(s_0, t_0, u_0)$};
		\node (II0) at (3, -1) {$G_0$};
		\node (I1) at (4, 0) {$(s_1, t_1, u_1)$};
		\node (II1) at (5, -1) {$G_1$};
		\node (I2) at (6, 0) {$(s_2, t_2, u_2)$};
		\node (II2) at (7, -1) {$G_2$};
		\node (dots) at (8, -0.5) {$\dots$};
		\draw[->] (I0) -> (II0);
		\draw[->] (II0) -> (I1);
		\draw[->] (I1) -> (II1);
		\draw[->] (II1) -> (I2);
		\draw[->] (I2) -> (II2);
	\end{tikzpicture}
	\]
	
	Assume that player I has a winning strategy $\sigma$.
	Put
	\[
	C = \{ (z, x, y) \in \omega^\omega \times \omega^\omega \times 2^\omega : (\exists (G_0, G_1, \dots))((z, x, y) \text{ is the play of I along } \sigma \text{ against } (G_0, G_1, \dots)) \}.
	\]
	Then clearly $C$ is a $\boldsig^1_1$ set.
	Since player I wins, we have $C \subset B$.
	So we have $\proj^3_0(C) \subset \proj^3_0(B) = A$.
	So each $(\proj^3_0(C))_x \subset A_x$ has measure $0$.
	For $x \in \omega^\omega$, put $D_x = \bigcup_{x' \le x} (\proj^3_0(C))_{x'}$, which is a $\boldsig^1_1$ set.
	Since $(\proj^3_0(C))_x \subset A_x$, each $D_x$ has measure $0$.
	And we have $x' \le x$ implies $D_{x'} \subset D_x$.
	
	Thus, by $\GP(\boldsig^1_1)$, $\proj^2_0(D)$ has measure $0$. So $\proj^2_0(\proj^3_0(C))$ has also measure $0$.
	Therefore we can take $(G_0, G_1, \dots)$ such that $\proj^2_0(\proj^3_0(C)) \subset \bigcup_{n \in \omega} G_n$ and $\mu(G_n) \le \varepsilon / 16^{n+1}$.
	
	Let $(z, x, y)$ be the play along $\sigma$ against $(G_0, G_1, \dots)$, then $(z, x, y) \in C$ and $y \not \in \bigcup_{n \in \omega} G_n$. This contradicts to $\proj^2_0(\proj^3_0(C)) \subset \bigcup_{n \in \omega} G_n$.
	
	So player I doesn't have a winning strategy. Therefore, by $\Det(\Gamma)$, player II has a winning strategy $\tau$.
	Put
	\[
	E = \bigcup \{ G_n : (G_0, \dots, G_n) \text{ is the play along } \tau \text{ against some } (s_0, t_0, u_0 \dots, s_n, t_n, u_n) \}.
	\]
	Then we have $\proj^2_0(\proj^3_0(B)) \subset E$.
	In order to check this, let $(z, x, y) \in B$.
	Consider the player I's play $(z, x, y)$.
	Let $(G_0, G_1, \dots)$ be the play along $\tau$ against $(z, x, y)$. 
	Since II wins, $y \in \bigcup_{n \in \omega} G_n \subset E$.
	
	Also we have
	\begin{align*}
		\mu(E) \le \sum_n 8^{n+1} \frac{\varepsilon}{16^{n+1}} = \varepsilon.
	\end{align*}
	Therefore we have $\mu^*(\proj^2_0(A)) \le \mu(E) \le \varepsilon$.
\end{proof}

\section{Consequences of large cardinals}\label{sec:largecardinals}

In this section, using large cardinals, we separate $\GP(\boldsig^1_{n+1})$ and $\GP(\boldsig^1_n)$ for every $n \ge 2$.

For a pointclass $\Gamma$, recall that $\triangleleft$ is a $\Gamma$-good wellordering of the reals if it is a wellordering of the reals of order-type $\omega_1$, it is in $\Gamma$ and the relation \[\{(x,y) : \text{$x$ codes the initial segment below $y$ with respect to $\triangleleft$} \}\] is in $\Gamma$.

\begin{fact}[{{\cite{bagaria1997sets}}} and {{\cite{STEEL199577}}}] \label{fact:measandwoodin}
	\begin{enumerate}
		\item If $\ZFC$ is consistent, then so is $\ZFC$ plus $\boldsig^1_2$ Lebesgue measurability plus ``there is a $\boldsig^1_3$ good wellordering of the reals of length $\omega_1$".
		\item Assume that there are $n$ many Woodin cardinals. Then there is an inner model $M_n$ of $\ZFC$ that models $\Det(\boldsig^1_n) $ and ``$\text{there is a $\boldsig^1_{n+2}$ good wellordering of the reals}$''.
	\end{enumerate}
\end{fact}

By an easy calculation, we can get the following lemma.

\begin{lem}\label{lem:goodwellorderingimpliesscale}
	Let $n \ge 2$.
	If there is a $\boldsig^1_n$ good wellordering $\trianglelefteq$ of the reals of length $\omega_1$, then there is a cofinal increasing sequence of $\omega^\omega$ whose image is $\bolddelta^1_n$. \qed
\end{lem}

Using the above lemma, we can get:

\begin{lem}\label{lem:goodwellorderingimpliesneggp}
	Let $n \ge 2$.
	If there is a $\boldsig^1_n$ good wellordering $\trianglelefteq$ of the reals of length $\omega_1$, then $\neg \GP(\bolddelta^1_n)$ holds.
\end{lem}
\begin{proof}
	Let $D$ denote the set defined in Lemma \ref{lem:goodwellorderingimpliesscale}.
	We define a set $A$ by
	\[
	A = \{ (x, y) \in \omega^\omega \times \omega^\omega : y \trianglelefteq z \text{ for the minimum $z \in D$ that dominates $x$} \}.
	\]
	Then we have
	\begin{align*}
		(x, y) \in A \iff & (\exists z)(\exists w) \\
		&\ [w \text{ codes the initial segment below } z \AND z \in D \ \AND \\
		&\ \ x \le^* z \ \AND (\forall k)((w)_k \in D \rightarrow x \not \le^* (w)_k) \ \AND \\
		&\ \ y \trianglelefteq z]
	\end{align*}
	So $A$ is $\boldsig^1_n$. Moreover, since we have
	\begin{align*}
		(x, y) \in A \iff & (\forall z)(\forall w) \\
		&\ [[w \text{ codes the initial segment below } z \AND z \in D \ \AND \\
		&\ \ x \le^* z \ \AND (\forall k)((w)_k \in D \rightarrow x \not \le^* (w)_k)] \ \rightarrow \\
		&\ \ y \trianglelefteq z],
	\end{align*}
	it is also true that $A$ is $\boldpi^1_n$. So $A$ is $\bolddelta^1_n$.
	
	Since each $A_x$ is countable and $\bigcup_{x \in \omega^\omega} A_x = \omega^\omega$, this $A$ witnesses $\neg \GP(\bolddelta^1_n)$.
\end{proof}

\begin{cor}
	\begin{enumerate}
		\item If $\ZFC$ is consistent, then so is $\ZFC+\GP(\boldsig^1_2)+\neg \GP(\bolddelta^1_3)$.
		\item For every $n \ge 1$, if $\ZFC+(\text{there are $n$ many Woodin cardinals})$ is consistent, then so is $\ZFC+\GP(\boldsig^1_{n+1})+\neg \GP(\bolddelta^1_{n+2})$.
	\end{enumerate}
\end{cor}
\begin{proof}
	As for (1), combine Fact \ref{fact:measandwoodin} (1), Theorem \ref{thm:measurabilityandgp} and Lemma \ref{lem:goodwellorderingimpliesneggp}.
	To show (2), combine Fact \ref{fact:measandwoodin} (2), Theorem  \ref{thm:determinacyimpliesgp} and Lemma \ref{lem:goodwellorderingimpliesneggp}.
\end{proof}

\section{$\GP(\all)$ in Solovay models}\label{sec:solovaymodel}

Now that we know that $\AD$ implies $\GP(\all)$, it is natural to ask whether $\GP(\all)$ holds in Solovay models.
In this section, we will solve this question affirmatively.

Basic information about Solovay models can be found in \cite[Chapter 3]{kanamori2008higher}.

Let us recall that $\Coll_\kappa$ denotes the Levy collapse.

\begin{defi}
	\begin{enumerate}
		\item $L(\R)^M$ is a Solovay model over $V$ (in the usual sense) if $M = V[G]$ for some inaccessible cardinal $\kappa$ and $(V, \Coll_\kappa)$ generic filter $G$.
		\item $L(\R)^M$ is a Solovay model over $V$ in the weak sense if the following 2 conditions hold in $M$:
		\begin{enumerate}
			\item For every $x \in \R$, $\omega_1$ is an inaccessible cardinal in $V[x]$.
			\item For every $x \in \R$, $V[x]$ is a generic extension of $V$ by some poset in $V$, which is countable in $M$.
		\end{enumerate}
	\end{enumerate}
\end{defi}

\begin{fact}[{{Woodin, see \cite[Lemma 1.2]{bagaria2004}}}]\label{fact:solovaymodelsweakandusual}
	If $L(\R)^M$ is a Solovay model over $V$ in the weak sense then there is a forcing poset $\mathbb{W}$ in $M$ such that $\mathbb{W}$ adds no new reals and
	\[
	\mathbb{W} \forces ``L(\R)^M \text{ is a Solovay model over } V \text{ (in the usual sense)}".
	\]
\end{fact}

\begin{fact}[{{\cite[Theorem 2.4]{bagaria2004}}}]\label{fact:genericextensionisalsosolovay}
	Suppose that $L(\R)^M$ is a Solovay model over $V$ in the weak sense and $\PP$ is a strongly-$\boldsig^1_3$ absolutely-ccc poset in $M$.
	Let $G$ be a $(M, \PP)$ generic filter.
	Then $L(\R)^{M[G]}$ is also a Solovay model in $V$ in the weak sense.
\end{fact}

We don't define the terminology ``strongly-$\boldsig^1_3$ absolutely-ccc poset" here. But the random forcing is such a poset and we will use only the random forcing when applying Fact \ref{fact:genericextensionisalsosolovay}.

\begin{lem}\label{lem:absoluteness}
	Let $M, N$ be models satisfying $V \subset M \subset N$.
	Assume that the $L(\R)$ of each of $M$ and $N$ are Solovay models over $V$ in the weak sense.
	Then for every formula $\varphi(v)$ in the language of set theory $\mathcal{L}_\in = \{\in\}$ and real $r$ in $M$, the assertion $``L(\R) \models \varphi(r)"$ is absolute between $M$ and $N$.
\end{lem}
\begin{proof}
	By Fact \ref{fact:solovaymodelsweakandusual}, we may assume that $L(\R)^M$ and $L(\R)^N$ are Solovay models over $V$ in the usual sense.
	By universality and homogeneity of the Levy collapse, we have
	\begin{align*}
		M \models ``L(\R) \models \varphi(r)" & \iff V[r] \models [\Coll_\kappa \forces ``L(\R) \models \varphi(r)"] \\
		& \iff N \models ``L(\R) \models \varphi(r)" 
	\end{align*}
\end{proof}

\begin{thm}\label{thm:solovaygpall}
	Let $\kappa$ be an inaccessible cardinal and $G$ be a $(V, \Coll_\kappa)$ generic filter.
	Then $L(\R)^{V[G]}$ satisfies $\GP(\all)$.
	That is, every Solovay model satisfies $\GP(\all)$.
\end{thm}
\begin{proof}
	Let $A \subset \omega^\omega \times 2^\omega$ in $L(\R)^{V[G]}$.
	Take a formula $\varphi$ and an ordinal $\alpha$ such that 
	\[A = \{ (x, y) : \varphi(\alpha, x, y) \}^{L(\R)^{V[G]}}.\]
	In $L(\R)^{V[G]}$, assume that
	\begin{enumerate}
		\item[(1)] $(\forall x \in \omega^\omega) (\exists c_x \in \omega^\omega) (c_x \text{ is a Borel code for a measure 0 set} \AND A_x \subset \hat{c_x})$
		\item[(2)] $(\forall x, x' \in \omega^\omega) (x \le x' \rightarrow A_x \subset A_{x'})$.
	\end{enumerate}
	Using the axiom of choice in $V[G]$ we can choose such a family $(c_x : x \in \omega^\omega)$. Note that this family is not necessarily in $L(\R)^{V[G]}$.
	
	Since every set of reals is measurable in $L(\R)^{V[G]}$, $\bigcup_{x \in \omega^\omega} A_x$ is measurable in $L(\R)^{V[G]}$.
	Now we assume that the measure is positive and take a closed code $d$ in $L(\R)^{V[G]}$ such that $\mu(\hat{d}) > 0$ and $\hat{d} \subset \bigcup_{x \in \omega^\omega} A_x$ in $L(\R)^{V[G]}$.
	
	Take a random real $r$ over $V[G]$ with $r \in \hat{d}$.
	Then by Lemma \ref{lem:absoluteness}, we have in $L(\R)^{V[G][r]}$
	\begin{enumerate}
		\item[(1')] $(\forall x \in \omega^\omega  \cap L(\R)^{V[G]}) (A_x \subset \hat{c_x})$, and
		\item[(2')] $(\forall x, x' \in \omega^\omega) (x \le x' \rightarrow A_x \subset A_{x'})$.
	\end{enumerate}
	By randomness, we have $r \not \in \hat{c_x}$ for all $x  \in \omega^\omega  \cap L(\R)^{V[G]}$.
	But (2') and the fact that the random forcing is $\omega^\omega$-bounding imply $r \not \in A_x$ for all $x \in \omega^\omega$ in $V[G][r]$.
	Thus we have $\hat{d} \setminus \bigcup_{x \in \omega^\omega} A_x \ne \emptyset$ in $L(\R)^{V[G][r]}$.
	Then using Lemma \ref{lem:absoluteness} again, we have $\hat{d} \setminus \bigcup_{x \in \omega^\omega} A_x \ne \emptyset$ in $L(\R)^{V[G]}$.
	It is a contradiction to the choice of $d$.
\end{proof}

\section{Other ideals than the Lebesgue measure zero ideal}\label{sec:otherideals}

In this section, we consider other ideals than the Lebesgue measure zero ideal.

\begin{defi}
	Let $\Gamma$ be a pointclass and $\I$ be an ideal on a Polish space $Y$.
	Then $\GP(\Gamma, \I)$ means the following statement:
	Let $A \subset \omega^\omega \times Y$ be in $\Gamma$.
	Assume the monotonicity condition for $A$.
	Assume also $A_x \in \I$ for every $x \in \omega^\omega$.
	Then $\bigcup_{x \in \omega^\omega} A_x$ is also in $\I$.
\end{defi}

Note that $\GP(\Borel, \meager)$ does not hold by considering $A_x = \{ y \in \omega^\omega : y \le^* x \}$.

In what follows, $\mathcal{E}$ means the $\sigma$-ideal generated by closed null sets. We would like to prove $\GP(\boldpi^1_1, \mathcal{E})$ does not hold.

To prove it, we recall the basic notions of interval partitions. Let $\IP$ be the set of all interval partitions.
Also $\sqsubset$ is the order on $\IP$ such that
\[
(I_k)_{k \in \omega} \sqsubset (J_n)_{n \in \omega} \iff (\forall^\infty n)(\exists k)(I_k \subset J_n).
\]

We use the following well-known fact.

\begin{fact}[{{\cite[Theorem 2.10]{blass2010combinatorial}}}]\label{fact:omegaomegaandip}
	There is a Borel Tukey connection from $\IP$ to $\omega^\omega$.
	That is there is Borel maps $\phi \colon \IP \to \omega^\omega$ and $\psi \colon \omega^\omega \to \IP$ such that $\phi(I) \le^* y$ implies $I \sqsubset \psi(y)$ for every $I \in \IP$ and $y \in \omega^\omega$. 
\end{fact}

	\begin{thm}
		$\GP(\boldpi^1_1, \mathcal{E})$ does not hold.
	\end{thm}
	\begin{proof}
		Let $\seq{I_n : n \in \omega}$ be the interval partition with $\abs{I_n} = n+1$.
		Let $A = \{ x \in 2^\omega : (\exists^\infty n)(x \upharpoonright I_n \equiv 0) \}$.
		It is well-known that $A$ is null and comeager. In particular $A \not \in \mathcal{E}$.
		For an interval partition $J = \seq{J_k : k \in \omega}$, let 
		\[A_J = \{ x \in 2^\omega : (\forall^\infty k)(\exists n \in J_k) x \upharpoonright I_n \equiv 0 \}.\]
		Each $A_J$ is $F_\sigma$ null set, so it holds that $A_J \in \mathcal{E}$.
		And also we have $\bigcup_{J} A_J = A \not \in \mathcal{E}$.
		
		Moreover, $\{ (J, y) : y \in A_J \}$ is Borel and the family is monotone.
		
		We now have to translate this family to the family indexed by $\omega^\omega$.
		
		Let $(\phi, \psi)$ be the Tukey connection in Fact \ref{fact:omegaomegaandip}.
		Let $B_x = \bigcap_{y \ge^* x} A_{\psi(y)}$ for $x \in \omega^\omega$.
		$\seq{B_x : x \in \omega^\omega}$ is clearly monotone, each $B_x$ is in $\mathcal{E}$ and the family $B$ is $\boldpi^1_1$.
		On the other hand we have $\bigcup_{J \in \IP} A_J \subset \bigcup_{x \in \omega} B_x$.
		To see it, let $J \in \IP$ and $a \in A_J$.
		Set $x = \phi(J)$.
		Then $x \le^* y$ implies $J \sqsubset \psi(y)$ since $(\phi, \psi)$ is Tukey. So $a \in A_J \subset A_{\psi(y)}$.
		Therefore we have $a \in B_x$.
		
		Since $\bigcup_{J \in \IP} A_J \not \in \mathcal{E}$, we have constructed a counterexample $\seq{B_x : x \in \omega^\omega}$ of $\GP(\boldpi^1_1, \mathcal{E})$.
	\end{proof}

\section{Open problems}

In this section, we list the problems that remain.

\begin{prob}\label{prob:largecontinuum}
	\begin{enumerate}
	\item Is $\ZFC+(\frakc>\aleph_2)+\GP(\all)$ consistent?
	\item Is $\ZFC+(\frakb<\frakd)+\GP(\all)$ consistent?
	\end{enumerate}
\end{prob}

We suspect that (1) of Problem \ref{prob:largecontinuum} is proved by adding many random reals over the Laver model.

\begin{prob}
	How much separation is possible between $\GP$ of point classes in the projection hierarchy? In particular how much is possible without large cardinals?
\end{prob}
Note that we separated $\GP(\boldsig^1_{n+1})$ and $\GP(\bolddelta^1_{n+2})$ using $n$ many Woodin cardinals. We are asking how about separation between $\GP(\boldsig^1_{n+1})$ and $\GP(\boldpi^1_{n+1})$? Also, we are asking whether we can eliminate the use of Woodin cardinals.

\begin{prob}
	Can $\GP(\bolddelta^1_2)$ and $(\forall a \in \R)(\exists z \in \omega^\omega) (\text{$z$ is a dominating real over $L[a]$})$ be separated?
\end{prob}

A nice candidate of this problem is Hechler model over $L$.

\begin{prob}\label{prob:lmplusnotgp}
Is there a model of $\ZF$ satisfying that every set of reals is measurable and $\neg \GP(\all)$?
\end{prob}

Shelah \cite{Sh:218} constructed a model of $\ZF$+(the dependent choice)+(every set of reals is Lebesgue measurable)+(there is a set of reals that does not have the Baire property). The idea of this construction might help Problem \ref{prob:lmplusnotgp}.

\begin{prob}
	Does $\GP(\all)$ imply the Borel conjecture?
\end{prob}
Note that the Borel conjecture does not imply $\GP(\all)$, since the Mathias model satisfies the Borel conjecture and $\non(\nul)=\frakb$, which implies $\neg \GP(\all)$.

Finally, we ask:

\begin{prob}
	Does $\GP(\boldsig^1_1, \mathcal{E})$ hold?
\end{prob}

	\section*{Acknowledgments}
	
	The author would like to thank Yasuo Yoshinobu, Takayuki Kihara and Jörg Brendle, who gave him advices on this research.
	In particular, the idea of the proof of Theorem \ref{lavermodelthm} is due to Brendle.
	He also would like to thank Martin Goldstern, who read the preprint of this paper and pointed out errors.
	The author also would like to thank Daisuke Ikegami, who gave him advice in proving Theorem \ref{thm:solovaygpall}.
	This work was supported by JSPS KAKENHI Grant Number JP22J20021.
	
	\printbibliography[title={References}]
\end{document}